\documentclass[10pt]{amsart}
\usepackage{amsmath,amsfonts,amsthm,amssymb,amsxtra,hyperref}
\usepackage{color,graphicx}\parskip 4pt
\newcommand{\Email}[1]{{\sl E-mail address:\/} {\rm\textsf{#1}}}
\newtheorem{thm}{Theorem}
\newtheorem{cor}[thm]{Corollary}
\newtheorem{lem}[thm]{Lemma}
\newtheorem{prop}[thm]{Proposition}

\newcommand{\R}{{\mathbb R}}
\newcommand{\Sphere}{{\mathbb S}}
\newcommand{\N}{{\mathbb N}}
\newcommand{\be}[1]{\begin{equation}\label{#1}}
\newcommand{\ee}{\end{equation}}
\renewcommand{\(}{\left(}
\renewcommand{\)}{\right)}
\newcommand{\eps}{\varepsilon}
\newcommand{\seq}[2]{({#1}_{#2})_{#2\in\N}}

\newcommand{\M}{\mathfrak M}
\newcommand{\iM}[1]{\int_{\M}{#1}\,d\kern1pt v_g}
\newcommand{\nrM}[2]{\|{#1}\|_{\L^{#2}(\M)}}
\newcommand{\Lap}{\Delta_g}
\newcommand{\Ric}{\mathfrak R}

\newcommand{\iS}[1]{\int_{\Sphere^d}{#1}\;d\kern1pt v_g}

\newcommand{\nrm}[2]{\|{#1}\|_{\L^{#2}(\M)}}

\renewcommand{\H}{\mathrm H}
\renewcommand{\L}{\mathrm L}

\usepackage{color}

\begin{document}

\title[Flows and rigidity on compact manifolds]{Nonlinear flows and rigidity results on compact manifolds}
\author[J.~Dolbeault, M.J.~Esteban \& M.~Loss]{Jean Dolbeault, Maria J.~Esteban, and Michael Loss}
\address{J.~Dolbeault: Ceremade UMR CNRS nr. 7534, UMR CNRS nr. 7534, Universit\'e Paris-Dauphine, Place de Lattre de Tassigny, 75775 Paris C\'edex~16, France.\vspace*{-6pt}
\Email{dolbeaul@ceremade.dauphine.fr}}
\address{M.J.~Esteban: Ceremade UMR CNRS nr. 7534, UMR CNRS nr. 7534, Universit\'e Paris-Dauphine, Place de Lattre de Tassigny, 75775 Paris C\'edex~16, France.\vspace*{-6pt}
\Email{esteban@ceremade.dauphine.fr}}
\address{M.~Loss: School of Mathematics, Skiles Building, Georgia Institute of Technology, Atlanta GA 30332-0160, USA.
\Email{loss@math.gatech.edu}}
\date{\today}
\begin{abstract} This paper is devoted to \emph{rigidity} results for some elliptic PDEs and to \emph{optimal constants} in related interpolation inequalities of Sobolev type on smooth compact connected Riemannian manifolds without boundaries. Rigidity means that the PDE has no other solution than the constant one at least when a parameter is in a certain range. The largest value of this parameter provides an estimate for the optimal constant in the corresponding interpolation inequality. Our approach relies on a nonlinear flow of porous medium / fast diffusion type which gives a clear-cut interpretation of technical choices of exponents done in earlier works on rigidity. We also establish two integral criteria for rigidity that improve upon known, pointwise conditions, and hold for general manifolds without positivity conditions on the curvature. Using the flow, we are also able to discuss the optimality of the corresponding constants in the interpolation inequalities.
\end{abstract}

\maketitle
\thispagestyle{empty}

%%%%%%%%%%%%%%%%%%%%%%%%%%%%%%%%%%%%%%%%%%%%%%%%%%%%%%%%%%%%%%%%%%%%%%
%%%%%%%%%%%%%%%%%%%%%%%%%%%%%%%%%%%%%%%%%%%%%%%%%%%%%%%%%%%%%%%%%%%%%%
\section{Introduction and main results}\label{Sec:Intro}

In the past decades there has been considerable activity in establishing sharp inequalities using maps or flows. The basic idea is to look for a flow on a function space along which a given functional converges to its optimal value, \emph{i.e.}, one turns the idea of a Lyapunov function, known from dynamical systems theory, on its head. An example is furnished by the relatively recent proofs of the Brascamp-Lieb inequalities using nonlinear heat flows in \cite{MR2377493,MR2448061,MR2077162}. Using the same methods a new Brascamp-Lieb type inequality on $\Sphere^d$ was proved in \cite{MR2077162}. Likewise the reverse Brascamp-Lieb inequalities can also be obtained in this fashion (see \cite{MR2087151}). Another example is the proof of Lieb's sharp Hardy-Littlewood-Sobolev inequality given in \cite{MR1038450} where a discrete map on a function space was constructed whose iterations drives the Hardy-Littlewood-Sobolev functional to its sharp value. Likewise, the sharp form of the Gagliardo-Nirenberg inequalities due to Dolbeault and del Pino can be derived using the porous media flow (see \cite{MR1777035}). The porous media equation can also be used in the context of a special class of Hardy-Littlewood-Sobolev inequalities (see~\cite{MR2745814}).

Closely related are the proofs of sharp inequalities using transportation theory. The earliest use of transportation theory to our knowledge was in Barthe's proof of the Brascamp-Lieb inequalities as well as their converse, in \cite{MR1650312}. Transportation ideas were also applied in \cite{MR2032031} for proving the sharp Gagliardo-Nirenberg inequalities and in \cite{MR2258478} for proving sharp trace inequalities.

In this paper we use a porous media flow on Riemannian manifolds that allow us to give relatively straightforward proofs as well as generalizations of rigidity results of \cite{MR1134481,MR1338283,MR1412446,MR1631581} for a class of nonlinear equations. Before describing the flow, we discuss the rigidity results.

Throughout the paper we assume that $(\M,g)$ is a smooth compact connected Riemannian manifold of dimension $d\ge1$, without boundary with $\Lap$ being the Laplace-Beltrami operator on $\M$. For simplicity, we assume that the volume of $\M$, $\mathrm{vol}(\M)$ is $1$ and we denote by $d\kern1pt v_g$ the volume element. We shall also denote by~$\Ric$ the Ricci tensor. Let~$\lambda_1$ be the lowest positive eigenvalue of $-\Lap$. We shall use the notation $2^*:=\frac{2\,d}{d-2}$ if $d\ge3$, and $2^*:=\infty$ if $d=1$ or $2$.

Let us start with results dealing with manifolds whose curvature is bounded from below and define
\[
\rho:=\inf_{\M}\inf_{\xi\in\Sphere^{d-1}}\Ric(\xi\,,\xi)\,.
\]
%---------------------------------------------------------------------
\begin{thm}\label{Thm:LV} Let $d\ge2$ be an integer and assume that $\rho$ is positive. If $\lambda$ is a positive parameter such that
\[
\lambda\le(1-\theta)\,\lambda_1+\theta\,\frac{d\,\rho}{d-1}\quad\mbox{where}\quad\theta=\frac{(d-1)^2\,(p-1) }{d\,(d+2)+p-1}\;,
\]
then for any $p\in(2,2^*)$, the equation
\be{Eqn}
-\,\Lap v+\frac\lambda{p-2}\,\(v-v^{p-1}\)=0
\ee
has a unique positive solution $v\in C^2(\M)$, which is constant and equal to $1$.
\end{thm}
%---------------------------------------------------------------------
Such a \emph{rigidity} result has been established in \cite{MR1338283} and \cite[Theorem 2.1]{MR1631581} by J.R.~Licois and L.~V\'eron, and in \cite[Inequality (1.11)]{MR1412446} by D.~Bakry and M.~Ledoux (also see \cite{MR1435336}). An earlier version of it can be found in a work of M.-F.~Bidaut-V\'eron and L.~V\'eron, \cite{MR1134481}, with an estimate only in terms of $\rho$, which is not as good for general manifolds but coincides with the one given above for spheres. A first version of such results can be found in~\cite{MR615628}. Each of these contributions relies either on the Bochner-Lichnerovicz-Weitzenb\"ock formula or on the \emph{carr\'e du champ} method, which are equivalent in the present case. We emphasize that all these results are based on estimates for nonlinear elliptic equations and so far no notion of flow has been associated to them. Computations and the choice of the exponents are rather technical. A first outcome of our approach is to give a simple setting for these issues.

The case of the critical exponent $p=2^*$ when $d\ge3$ is not covered by Theorem~\ref{Thm:LV} and requires more care, due to compactness issues and conformal invariance. For simplicity we shall restrict ourselves to the subcritical case in this paper, at least as far as rigidity results are concerned.

We shall now state a slightly more general result, which does not require that $\rho$ is positive but only involves the constant
\[
\lambda_\star:=\inf_{u\in\H^2\,(\M)}\frac{\displaystyle\iM{\Big[(1-\theta)\,(\Lap u)^2+\frac{\theta\,d}{d-1}\,\Ric(\nabla u,\nabla u)\Big]}}{\iM{|\nabla u|^2}}\;.
\]
Here $\theta$ is defined as in Theorem~\ref{Thm:LV} and we will relate the functions $u$ in the above quotient with the solution $v$ of \eqref{Eqn} using a nonlinear flow below in this section. If $\rho$ is positive, it is not difficult to check (see Lemma~\ref{Lem:FirstEigenvalue}) that
\[
\lambda_\star\ge(1-\theta)\,\lambda_1+\theta\,\frac{d\,\rho}{d-1}\;.
\]
However, we shall simply assume that $\lambda_\star$ is positive but make no assumption on the sign (in the sense of quadratic forms) of the Ricci tensor $\Ric$. Note that as soon as the Ricci curvature is not constant, the above inequality is strict. Our first result generalizes Theorem~\ref{Thm:LV} as follows.
%---------------------------------------------------------------------
\begin{thm}\label{Thm:Main} With the above notations, if $\lambda$ is such that
\[
0<\lambda<\lambda_\star\;,
\]
then for any $p\in(1,2)\cup(2,2^*)$, Equation~\eqref{Eqn} has a unique positive solution in $C^2(\M)$, which is constant and equal to $1$.
\end{thm}
%---------------------------------------------------------------------
In this statement we include the case $p\in(1,2)$. We may observe that $\lim_{p\to1_+}\theta(p)=0$, so that $\lim_{p\to1_+}\lambda_\star(p)=\lambda_1$, as soon as $\rho$ is bounded (but eventually negative). Notice that if $\M=\Sphere^d$, then $\lambda_\star=\lambda_1=d\,\rho/(d-1)=d$ since $\rho=d-1$. In this case, the result of Theorem~\ref{Thm:LV} is then optimal, but this was already known from \cite[Theorem 6.1]{MR1134481}. As we shall see in Proposition~\ref{Prop:Estimates}, we always have
\[
\lambda_\star\le\lambda_1
\]
and \eqref{Eqn} has non-constant solutions for any $\lambda>\lambda_1$. We can actually give another integral criterion for rigidity that slightly improves on the result of Theorem~\ref{Thm:Main}. Recalling that~$\theta$ is given as in Theorem~\ref{Thm:LV} by
\[
\theta=\frac{(d-1)^2\,(p-1)}{d\,(d+2)+p-1}\;,
\]
let us define
\[
\mathrm Q_g u:=\mathrm H_gu-\frac gd\,\Delta_gu-\frac{(d-1)\,(p-1)}{\theta\,(d+3-p)}\left[\frac{\nabla u\otimes\nabla u}u-\frac gd\,\frac{|\nabla u|^2}u\right]
\]
where $\mathrm H_gu$ denotes Hessian of $u$, and define
\be{LambdaStar}
\Lambda_\star:=\inf_{u\in\H^2(\M)\setminus\{0\}}\frac{\displaystyle(1-\theta)\iM{(\Delta_gu)^2}+\frac{\theta\,d}{d-1}\iM{\Big[\,\|\mathrm Q_g u\|^2+\Ric(\nabla u,\nabla u)\Big]}}{\displaystyle\iM{|\nabla u|^2}}\;.
\ee
It is obvious that
\[
\lambda_\star\le\Lambda_\star\;,
\]
and it is not hard to see that there is again equality in case of the sphere.
%---------------------------------------------------------------------
\begin{thm}\label{Thm:Main2} Assume that $\Lambda_\star>0$. For any $p\in(1,2)\cup(2,2^*)$, Equation~\eqref{Eqn} has a unique positive solution in $C^2(\M)$ if $\lambda\in(0,\Lambda_\star)$, which is constant and equal to $1$.\end{thm}
%---------------------------------------------------------------------
As we shall see below, our approach of rigidity results relies on a nonlinear diffusion equation of porous media (fast diffusion) type. Its interest is that it simplifies the choice of the exponents, introduces non-local quantities and simplifies the discussion of the optimality of the constants in the associated functional inequalities. It also applies to equations which are not of power law type (see Theorem~\ref{Thm:General}) and provides an expression of the error term (see Corollary~\ref{Cor:InequalityRemainder}) in the interpolation inequality
\be{Ineq:Interp}
\nrm{\nabla v}2^2\ge\frac\lambda{p-2}\,\left[\nrm vp^2-\nrm v2^2\right]\quad\forall\,v\in\H^1(\M)\;.
\ee
Notice that this inequality makes sense not only for $p\in(2,2^*)$ but also for $p\in(1,2)$ because in this last case we have $\nrm vp\le\nrm v2$ by H\"older's inequality and $p-2$ is negative. When $p>2$, Inequality~\eqref{Ineq:Interp} has been stated in \cite[Corollary 6.2]{MR1134481} with $\lambda=d\,\rho/(d-1)$, but has also been established in the case of the sphere in \cite{MR1230930} with $\lambda=d$ using spectral estimates and Lieb's sharp Hardy-Littlewood-Sobolev inequality (see~\cite{MR717827}). The case $p\in(1,2)$ has been covered in \cite[Section~3.11]{Demange-PhD}. We shall see in Proposition~\ref{Prop:Estimates} that $\Lambda_\star\le\lambda_1$. As a consequence of the proof of Theorem~\ref{Thm:Main2}, we get also the following result.
%---------------------------------------------------------------------
\begin{thm}\label{Thm:Inequality} Assume that $\Lambda_\star>0$. For any $p\in[1,2)\cup(2,2^*]$ if $d\ge3$, $p\in(1,2)\cup(2,\infty)$ if $d=1$ or~$2$, Inequality~\eqref{Ineq:Interp} holds for some $\lambda=\Lambda\in[\Lambda_\star,\lambda_1]$. Moreover, if $\Lambda_\star<\lambda_1$ and $p<2^*$, then the optimal constant $\Lambda$ is such that
\[
\Lambda_\star<\Lambda\le\lambda_1\;.
\]
If $p=1$, then $\Lambda=\lambda_1$.
\end{thm}
%---------------------------------------------------------------------
While the case $p=2^*$ is not covered in Theorem~\ref{Thm:Main2}, Inequality~\eqref{Ineq:Interp} holds in that case by taking the limit $p\to2^*_-$ when $d\ge3$. In the limit case $p=1$, Inequality~\eqref{Ineq:Interp} is nothing else but the Poincar\'e inequality which is consistent with the fact that $\lim_{p\to1_+}\Lambda_\star(p)=\lim_{p\to1_+}\lambda_\star(p)=\lambda_1$. Using $u=1+\eps\,\varphi$ as a test function where $\varphi$ is an eigenfunction associated with the lowest positive eigenvalue of $-\Lap$, it is a classical result that the optimal constant in~\eqref{Ineq:Interp} is $\lambda\le\lambda_1$. For the same reason, a minimum of
\[
v\mapsto\nrm{\nabla v}2^2-\frac\lambda{p-2}\,\left[\nrm vp^2-\nrm v2^2\right]\,,
\]
under the constraint $\nrm vp=1$ is negative if $\lambda>\lambda_1$. Incidentally, such a minimization provides a non-constant non-negative solution of \eqref{Eqn} (see \cite{MR2252974,DoEsLa2012}). It is natural to wonder what happens when $p=2$. The reader is invited to check that by taking the limit as $p\to2$, one gets a logarithmic Sobolev inequality, which is of independent interest (see for instance \cite{MR1961176} or \cite{DoEsLa2012} for more results in this direction and further references).

\medskip Now let us give some indications on our method. In this paper we consider the flow
\be{flow}
u_t=u^{2-2\,\beta} \(\Delta_gu+\kappa\,\frac{|\nabla u|^2}u\)\,,
\ee
where we set
\be{Eqn:kappa}
\kappa=1+\beta\,(p-2)\;.
\ee
The flow~\eqref{flow} contracts both sides of Inequality~\eqref{Ineq:Interp} and monotonically decreases
\be{functional}
\mathcal F[u]:=\iM{|\nabla(u^\beta)|^2}+\frac{\lambda}{p-2}\left[\iM{u^{2\,\beta}}-{\(\iM{u^{\beta\,p}}\)}^{2/p}\right]
\ee
if $\lambda<\Lambda_\star$. As we shall see in Sections~\ref{Sec:Flow} and~\ref{Sec:Proofs}, the limit of $t\mapsto\mathcal F[u(t,\cdot)]$ is zero, which proves the inequality $\mathcal F[u(t,\cdot)]\ge0$ for any $t\ge0$ and in particular for the initial datum $u(t=0,\cdot)=v^{1/\beta}$. The key ingredient to prove the monotonicity of $t\mapsto\mathcal F[u(t,\cdot)]$ is the inequality
\be{Ineq:New}
(1-\theta)\iM{(\Delta_gu)^2}+\frac{\theta\,d}{d-1}\iM{\Big[\,\|\mathrm Q_g u\|^2+\Ric(\nabla u,\nabla u)\Big]}\ge\Lambda_\star\,\iM{|\nabla u|^2}
\ee
written for any $u=u(t,\cdot)$, $t\ge0$. Equality in the inequality $\frac d{dt}\mathcal F[u(t,\cdot)]\le0$ is achieved only by constant functions. Since $\mathcal F[v^{1/\beta}]=0$ if $v$ is a solution to \eqref{Eqn}, this also shows the rigidity result of Theorem~\ref{Thm:Main}. The parameter $\beta$ in the relation
\[
v=u^\beta
\]
is chosen in such a way that the problem is reduced to \eqref{Ineq:New}, which is much easier to establish than \eqref{Ineq:Interp} because of its homogeneity. For this reason, the choice of $\beta$ is somewhat more transparent than in \cite{MR1134481,MR1631581} (see Section~\ref{Sec:Flow}) and allows for simple generalizations, like the one stated in Theorem~\ref{Thm:General} concerning more general nonlinearities than just power laws. The choice of $\kappa$ is determined by the condition that $\nrm vp$ is invariant under the action of the flow (see Lemma~\ref{Lem:Lp}). The equation \eqref{flow} is of the type of a porous media equation. Let us note in this context that Hamilton's Yamabe flow can be seen as the fast diffusion case of a porous media equation \cite{MR1258912} but it is not clear whether there is any connection between this and \eqref{flow}.

The flow defined by \eqref{flow} has already been used by J.~Demange. In \cite{MR2178944}, he studies critical functional inequalities on manifolds with negative Ricci curvature; such issues are out of the focus of the present paper. In \cite{MR2381156}, he studies the case of a pointwise positive curvature when $p>2$ and the flow he considers is equivalent to that of \eqref{flow}. His computations are done in the $\Gamma_2$ formalism (or \emph{carr\'e du champ} method), cover the subcritical case, and by some additional interpolation estimates allow him to get nonlinear improvements of the inequalities. However, because of the pointwise positivity condition, optimality of the constants even in the leading order term is definitely out of reach in Demange's approach, except in the case of the sphere. For instance, the improved estimates that can be found in \cite[Inequality (1.11)]{MR1412446} are not even recovered. A simpler version of these results in the critical case can be found in C.~Villani's book \cite[Theorems~24.2~(iii) and 25.3, \S~25.5]{MR2459454}. In both of the two presentations, the approach is centered on the decay properties of the various terms involved in the functional $\mathcal F$ when evolved according to~\eqref{flow}. By doing so, improved rates of decay (far from the asymptotic regime) as well as nonlinear functional inequalities are achieved but clearly miss optimality results and the connection with the usual machinery of rigidity theorems. This is the point that we want to emphasize and which is the main contribution of this paper.

In particular, we clarify the following point about Theorem~\ref{Thm:LV}. In \cite{MR1134481}, it is established in a weaker form which corresponds to $\theta=1$. With this choice the flow defined by \eqref{flow} decreases the functional $\mathcal F$ for a whole interval of values of the parameter $\beta$. This was observed in \cite[Section~3.12]{Demange-PhD} with a condition which is not fully explicit. It is, however, stated with accurate values in \cite{DEKL}, for the case of the sphere and we will give in Corollary~\ref{Cor:Theta=1} a fully explicit statement on the admissible range for the parameter $\beta$, valid for more general manifolds. While Demange obtained a nonlinear improvement of the inequality, his constant in front of the leading order term can be improved as has been done in \cite{MR1338283,MR1412446,MR1631581} by taking~$\theta$ different from~$1$. The value of $\theta$ which is used throughout our paper is the one which optimizes the rigidity interval when $\beta$ varies.

This paper is organized as follows. In Section~\ref{Sec:Two}, we introduce two simple computations, which differ from the ones that can be found in \cite{MR1134481,MR1338283,MR1412446,MR1631581}: results are stated in Lemmata~\ref{BW1} and~\ref{BW2}. Section~\ref{Sec:Flow} is devoted to the study of the evolution of the functional $\mathcal F$ given by \eqref{functional} under the action of the flow defined by~\eqref{flow}. Proofs of Theorem~\ref{Thm:Main2} and Theorem~\ref{Thm:Inequality} are given in Section~\ref{Sec:Proofs}. Moreover, a simple inspection of these proofs allows to state an improved version of Theorem~\ref{Thm:Inequality}, with an integral remainder term. The question of the optimality of the range of rigidity and of the optimal constant in the interpolation inequality~\eqref{Ineq:Interp} will also be discussed in Section~\ref{Sec:Proofs}. The use of the flow~\eqref{flow} shows that our estimates are in general better than the ones known in the existing literature but still not optimal, although our approach relies only on global estimates (rather than on pointwise estimates). We stress the fact that global estimates allow us to consider the case of sign changing curvatures, an issue that has apparently not been considered up to now. Finally in Section~\ref{Sec:CNS} some extensions of the above results will be presented and some concluding remarks can be found in Section~\ref{Sec:Conclusion}.

\medskip Before dealing with our method, let us give a brief additional list of papers which are relevant for our approach. This is not an easy task because there are clearly several lines of thoughts which have developed in parallel. We will only mention the papers which are closely related to our purpose and select in the huge literature corresponding to the various topics listed below the earliest contributions we are aware of, or some recent papers to which we refer for more details on the historical developments. The motivation for this work comes from the method developed by B.~Gidas and J.~Spruck in \cite[Theorems B.1 and B.2]{MR615628} for proving \emph{rigidity} and later improved first by M.-F.~Bidaut-V\'eron and L.~V\'eron in \cite[Theorem 6.1]{MR1134481}, and then by J.R.~Licois and L.~V\'eron in \cite{MR1338283} and \cite[Theorem 2.1]{MR1631581}, and by D.~Bakry and M.~Ledoux in \cite[Inequality (1.11)]{MR1412446}. The choice of the exponents was somewhat mysterious in these papers while the choice of $\beta$ and $\kappa$ is rather straightforward and natural in our setting.

When $\M=\Sphere^d$, the \emph{rigidity} result with optimal range in terms of the parameter $\lambda$ was established in~\cite{MR1134481} and later, by other methods, in \cite{MR1230930}. The case of the critical exponent was known much earlier since the inequality corresponds to the Sobolev inequality on the Euclidean space, after a stereographic projection: we can refer to \cite[Corollaire 4]{MR0431287}, to \cite{MR0303464} for issues dealing with conformal invariance, to \cite{DEKL} for results on the sphere, and to \cite{DoEsLa2012} for consequences in the spectral theory of Schr\"odinger operators.

For general manifolds, the so-called $A$-$B$ problem, which amounts to determine the two best constants in interpolation inequalities, has attracted the attention of many authors and it is out of the scope of this paper to list all of them. Let us simply quote \cite{MR0448404,MR534672,MR608289,MR667236,hebey1994meilleures,MR2592965} as a selection of contributions on the ``second best constant'' problem. The interested reader can also refer to \cite{MR1688256} for a large overview of the topic. With some exceptions (see for instance \cite{MR1961176,MR2273884}), almost all related papers are concerned either with the case of critical exponents and related concentration issues or with the limiting case $p=2$, where Inequality~\eqref{Ineq:Interp} has to be replaced by the logarithmic Sobolev inequality (see \cite{federbush1969partially,Gross75,MR620582,MR674060} for early contributions to this topic).

Inequality~\eqref{Ineq:Interp} has been studied for $p\in(1,2)$ and various probability measures, mostly defined on the Euclidean space: we may refer for instance to \cite{MR772092,Bakry-Emery85,MR1796718,MR2081075,MR2273884,MR2609029}. The paper \cite{MR2273884} by P.~Deng and F.~Wang is of particular interest, as it gives estimates of the constant $\lambda$ in~\eqref{Ineq:Interp} for an arbitrary $p\in(1,2)$ using several methods, including in the case of compact manifolds and without assuming the positivity of the curvature. In the particular case of the sphere, the whole range $p\in(1,2^*]$ was covered in the case of the ultraspherical operator in \cite{MR1231419,MR2564058}. Also see \cite{DEKL} for the use of the ultraspherical operator in the context of \emph{rigidity} results.

Using flows on manifolds is a classical idea, including flows associated either to the heat equation or to porous media equations, or in relation with Ricci (see for instance \cite{MR2061425}) and Yamabe flows (see \cite[Appendix III]{MR2282669} for an introduction). Another line of thought based on flows refers to the \emph{carr\'e du champ} method: see \cite{MR772092,MR1307413,our-ls-2000,MR2213477} and \cite{MR1813804} for a quite detailed review by M.~Ledoux. In the framework of diffusion equations from the PDE point of view we primarily refer to \cite{AMTU} in the linear case, and to \cite{MR1777035,MR1940370,MR1986060,MR1853037} for fast diffusion equations in the Euclidean space. This approach was later interpreted in terms of gradient flows with respect to Wasserstein's distance in~\cite{MR1842429}, and then extended or reinterpreted in various contributions in the context of \emph{entropy methods}. Such a direction of research is out of the scope of the present paper, although it has been a strong source of inspiration. A very extended account can be found in \cite{MR2459454}. For completeness, let us mention that the flow defined by \eqref{flow} has been considered for instance in~\cite{DEKL} when $\beta=1$, or when $\kappa=0$ and either $p=2$ of $p=2^*$. When $\beta\neq1$, an equivalent form written for $\rho=u^{\beta\,p}$ is the main tool in \cite{Demange-PhD,MR2381156} (also see \cite{MR2487898}) and, as already mentioned above, one can refer to~\cite{MR2459454} for the critical case.

%%%%%%%%%%%%%%%%%%%%%%%%%%%%%%%%%%%%%%%%%%%%%%%%%%%%%%%%%%%%%%%%%%%%%%
%%%%%%%%%%%%%%%%%%%%%%%%%%%%%%%%%%%%%%%%%%%%%%%%%%%%%%%%%%%%%%%%%%%%%%
\section{Some preliminary computations}\label{Sec:Two}

If $\mathrm A_{i,j}$ and $\mathrm B_{i,j}$ are two tensors we use the notation
\[
\mathrm A\cdot \mathrm B:=g^{i,m}\,g^{j,n}\,\mathrm A_{i,j}\,\mathrm B_{m,n}\quad\mbox{and}\quad\|\mathrm A\|^2:=\mathrm A\cdot \mathrm A\;.
\]
Here $g^{i,j}$ is the inverse of the metric tensor, \emph{i.e.}, $g^{i,j}\,g_{j,k}=\delta^i_k$ where we used the Einstein summation convention and $\delta^i_k$ denotes the Kronecker symbol. Denote by $\mathrm H_g u$ the Hessian of $u$ and by
\[
\mathrm L_gu:=\mathrm H_gu-\frac gd\,\Delta_gu
\]
the trace free Hessian. The following lemma is well known, but since it is of some importance for what follows we give a short proof.
%---------------------------------------------------------------------
\begin{lem}\label{BW1} If $d\ge 2$ and $u\in C^2\,(\M)$, then we have
\[\label{modifiedBW}
\iM{(\Delta_gu)^2}=\frac d{d-1}\iM{\|\,\mathrm L_gu\,\|^2}+\frac d{d-1}\iM{{\Ric}(\nabla u,\nabla u)}\;.
\]
\end{lem}
%---------------------------------------------------------------------
\begin{proof} Integrating the Bochner-Lichnerowicz-Weitzenb\"ock formula
\[
\frac12\,\Delta\,|\nabla u|^2=\|\mathrm H_gu\|^2+\nabla (\Delta_gu)\cdot\nabla u+\Ric(\nabla u,\nabla u)
\]
yields
\[
\iM{(\Delta_gu)^2}=\iM{\|\mathrm H_gu\|^2}+\iM{\Ric(\nabla u,\nabla u)}\;,
\]
where, incidentally, we may remark that the right-hand side is nothing else than the integral of $\Gamma_2(u,u)$ in the \emph{carr\'e du champ} method: see for instance \cite{MR1412446}. Next we notice that
\[
[\mathrm L_gu]\cdot g=0\;,
\]
because $\Delta_gu$ is the trace of $\mathrm H_gu$, which implies
\[
\iM{\|\,\mathrm L_gu\,\|^2}=\iM{\|\mathrm H_gu\|^2}-\frac1d\iM{(\Delta_gu)^2}\;,
\]
{}from which Lemma~\ref{BW1} follows.\end{proof}
The next lemma will be useful to express some important quantities as a square. Here we deal with the cross term.
%---------------------------------------------------------------------
\begin{lem}\label{BW2} If $u\in C^2\,(\M)$ is a positive function, then we have
\[\label{aux}
\iM{\Delta_gu\,\frac{|\nabla u|^2}u}=\frac d{d+2}\iM{\frac{|\nabla u|^4}{u^2}}-\frac{2\,d}{d+2}\iM{[\mathrm L_gu]\cdot\left[\frac{\nabla u\otimes\nabla u}u\right]}\;.
\]
\end{lem}
%---------------------------------------------------------------------
\begin{proof} An integration by parts shows that
\[
\iM{\Delta_gu\,\frac{|\nabla u|^2}u}=\iM{\frac{|\nabla u|^4}{u^2}}-2\iM{[\mathrm H_g u]\cdot \left[\frac{\nabla u \otimes \nabla u}u\right]}\,,
\]
which equals
\[
\iM{\frac{|\nabla u|^4}{u^2}}-2\iM{[\mathrm L_gu] \cdot \left[\frac{\nabla u \otimes \nabla u}u\right]}
-\frac2d\iM{\Delta_gu\,\frac{|\nabla u|^2}u}\;,
\]
by definition of $\mathrm L_gu$. Lemma~\ref{BW2} follows.\end{proof}

Recall that $(\M,g)$ is a smooth compact connected Riemannian manifold without boundary and consider the Poincar\'e inequality
\be{Ineq:Poincare}
\iM{|\nabla u|^2}\ge\lambda_1\iM{|u-\overline u|^2}\quad\forall\,u\in\H^1(\M)\quad\mbox{such that}\quad\overline u=\iM u\;,
\ee
with optimal constant $\lambda_1>0$. By standard compactness results, it is known that $\lambda_1$ can be characterized as the minimum of $\iM{|\nabla u|^2}/\iM{|u-\overline u|^2}$ on all non-constant functions $u\in\H^1(\M)$ and is achieved by a function $\varphi\in\H^1(\M)$ such that $\nrM\varphi2=1$ and $\overline\varphi=0$.

The third lemma of this section is a standard result, that is nonetheless crucial in our approach. For completeness, let us give a statement and a short proof.
%---------------------------------------------------------------------
\begin{lem}\label{Lem:FirstEigenvalue} With the above notations, we have
\[
\iM{(\Delta_gu)^2}\ge\lambda_1\iM{|\nabla u|^2}\quad\forall\,u\in\H^2(\M)\;.
\]
Moreover, $\lambda_1$ is the optimal constant in the above inequality.
\end{lem}
%---------------------------------------------------------------------
\begin{proof} A simple computation based on the Cauchy-Schwarz inequality
\[
{\(\iM{|\nabla u|^2}\)}^2={\(\iM{\Lap u\;(u-\overline u)}\)}^2\le\iM{(\Lap u)^2}\iM{|u-\overline u|^2}
\]
shows that
\[
\iM{(\Lap u)^2}\ge\frac{\(\iM{|\nabla u|^2}\)^2}{\iM{|u-\overline u|^2}}
\]
and one observes using \eqref{Ineq:Poincare} that
\[
\frac{\(\iM{|\nabla u|^2}\)^2}{\iM{|u-\overline u|^2}}\ge\lambda_1\iM{|\nabla u|^2}\;,
\]
which concludes the proof. Notice that optimality in the Cauchy-Schwarz inequality means that $\Lap u$ and $(u-\overline u)$ are collinear, that is, $(u-\overline u)$ is an eigenfunction. It is then clear that equality holds if $u=\varphi$, an eigenfunction associated with the first positive eigenvalue $\lambda_1$.
\end{proof}

As a consequence of Lemmata~\ref{BW1} and~\ref{Lem:FirstEigenvalue}, we find that
\[
\lambda_1=\inf\frac{\iM{(\Delta_gu)^2}}{\iM{|\nabla u|^2}}\ge\frac d{d-1}\inf\frac{\iM{{\Ric}(\nabla u,\nabla u)}}{\iM{|\nabla u|^2}}\;,
\]
thus establishing the Lichnerowicz estimate: $\lambda_1\ge\frac{d\,\rho}{d-1}$ (see \cite[p.~135]{MR0124009} for the original statement, and \cite[p.~Ê179]{MR0282313} for the Obata-Lichnerowicz theorem).

%%%%%%%%%%%%%%%%%%%%%%%%%%%%%%%%%%%%%%%%%%%%%%%%%%%%%%%%%%%%%%%%%%%%%%
%%%%%%%%%%%%%%%%%%%%%%%%%%%%%%%%%%%%%%%%%%%%%%%%%%%%%%%%%%%%%%%%%%%%%%
\section{A porous media flow}\label{Sec:Flow}

This section is devoted to the study of the functional defined in \eqref{functional} and its evolution under the action of the flow defined by \eqref{flow}, which is an equation of \emph{porous media} type (the exponent $\beta$ is in some cases smaller than $1$ and one should then speak of \emph{fast diffusion}: see discussion at the end of this section). We will expose the computations without taking care of regularity issues that are requested to justify integrations by parts and give afterwards a sketch of the proofs and regularizations that are eventually needed. The first statement explains how the coefficient $\kappa$ is chosen.
%---------------------------------------------------------------------
\begin{lem}\label{Lem:Lp} With $\kappa$ given by \eqref{Eqn:kappa}, the functional $u\mapsto\iM{u^{\beta\,p}}$ remains constant if $u$ is a smooth positive solution of~\eqref{flow}.\end{lem}
%---------------------------------------------------------------------
\begin{proof} Computing the time derivative, we get that
\[
\frac d{dt}\iM{u^{\beta\,p}}=\beta\,p\iM{u^{1+\beta\,(p-2)}\(\Delta_gu+\kappa\,\frac{|\nabla u|^2}u\)}=0
\]
by the choice of $\kappa$.\end{proof}

Next we compute the time derivative of the remaining terms in the functional \eqref{functional}.
%---------------------------------------------------------------------
\begin{lem}\label{Lem:F} If if $u$ is a smooth positive solution of~\eqref{flow}, the following identity holds
\[
\frac1{2\,\beta^2}\,\frac d{dt}\mathcal F[u]=-\iM{\left[ (\Delta_gu)^2+(\kappa+\beta-1)\,\Delta_gu\,\frac{|\nabla u|^2}u+\kappa\,(\beta-1)\,\frac{|\nabla u|^4}{u^2}\right]}+\lambda\iM{|\nabla u|^2}\;.
\]
\end{lem}
%---------------------------------------------------------------------
\begin{proof} By Lemma~\ref{Lem:Lp}, we only have to compute $\frac d{dt}\iM{|\nabla(u^\beta)|^2}$ and $\frac d{dt}\iM{u^{2\,\beta}}$. We get
\begin{eqnarray*}
\frac12\,\frac d{dt}\iM{|\nabla(u^\beta)|^2}&=&-\,\beta\iM{\Delta_g(u^\beta)\,u^{1-\beta}\,\(\Delta_gu+\kappa\,\frac{|\nabla u|^2}u\)}\\
&=&-\,\beta^2\iM{\(\Delta_gu+(\beta-1)\,\frac{|\nabla u|^2}u\)\(\Delta_gu+\kappa\,\frac{|\nabla u|^2}u\)}\\
&=&-\,\beta^2\iM{\left[(\Delta_gu)^2+(\kappa+\beta-1)\,\Delta_gu\,\frac{|\nabla u|^2}u+\kappa\,(\beta-1)\,\frac{|\nabla u|^4}{u^2}\right]}
\end{eqnarray*}
and
\begin{eqnarray}\nonumber
\frac12\,\frac d{dt}\iM{u^{2\,\beta}}&=&\beta\iM{u^{2\,\beta-1}\,u^{2-2\,\beta}\,\(\Delta_gu+\kappa\,\frac{|\nabla u|^2}u\)}\\
\label{Eqn:L2}&=&\beta\,(\kappa-1)\iM{|\nabla u|^2}\\
&=&\beta^2\,(p-2)\iM{|\nabla u|^2}\;,\nonumber
\end{eqnarray}
which completes the proof.\end{proof}

Notice that for any $\theta\in(0,1]$ we can write the result of Lemma~\ref{Lem:F} as
\[
\frac1{2\,\beta^2}\,\frac d{dt}\mathcal F[u]=-\,(1-\theta)\iM{(\Delta_gu)^2}-\mathcal G[u]+\lambda\iM{|\nabla u|^2}\;,
\]
where
\[
\mathcal G[u]:=\iM{\left[\theta\,(\Delta_gu)^2+(\kappa+\beta-1)\,\Delta_gu\,\frac{|\nabla u|^2}u+\kappa\,(\beta-1)\,\frac{|\nabla u|^4}{u^2}\right]}\;.
\]
We shall now introduce a key quantity for our method. For any positive function $u\in C^2(\M)$ set
\[
\mathrm Q_g^\theta u:=\mathrm L_gu-\frac 1\theta\,\frac{d-1}{d+2}\,(\kappa+\beta-1)\left[\frac{\nabla u\otimes\nabla u}u-\frac gd\,\frac{|\nabla u|^2}u\right]\,.
\]
Notice that $\mathrm Q_g^\theta$ is a matrix valued, traceless quantity. The exponent $\beta$ in the flow~\eqref{flow} and the value of $\theta\in(0,1]$ will be chosen below in this section.
%---------------------------------------------------------------------
\begin{lem}\label{Lem:G} Assume that $d\ge 2$. With the above notations, any positive function $u\in C^2(\M)$ satisfies the identity
\[
\mathcal G[u]=\frac{\theta\,d}{d-1}\left[\iM{\|\mathrm Q_g^\theta u\|^2}+\iM{{\Ric}(\nabla u,\nabla u)}\right]-\,\mu\iM{\frac{|\nabla u|^4}{u^2}}
\]
with $\mu:=\frac1\theta\big(\frac{d-1}{d+2}\big)^2(\kappa+\beta-1)^2-\kappa\,(\beta-1)-(\kappa+\beta-1)\,\frac d{d+2}$.\end{lem}
%---------------------------------------------------------------------
\begin{proof} By Lemma \ref{BW1} and \ref{BW2}, we know that
\begin{multline*}\hspace*{-10pt}
\mathcal G[u]=\frac{\theta\,d}{d-1}\left[\iM{\|\,\mathrm L_gu\,\|^2}+\iM{{\Ric}(\nabla u,\nabla u)}\right]\\
+(\kappa+\beta-1)\left[\frac d{d+2}\iM{\frac{|\nabla u|^4}{u^2}}
-\frac{2\,d}{d+2}\iM{[\mathrm L_gu]\cdot\left[\frac{\nabla u\otimes\nabla u}u\right]}\right]+\kappa\,(\beta-1)\iM{\frac{|\nabla u|^4}{u^2}}
\end{multline*}
which can be arranged into
\begin{multline*}
\mathcal G[u]=\frac{\theta\,d}{d-1}\left[\iM{\|\,\mathrm L_gu\,\|^2}-\frac2\theta\,\frac{d-1}{d+2}\,(\kappa+\beta-1)\iM{[\mathrm L_gu]\cdot\left[\frac{\nabla u\otimes\nabla u}u\right]}
\right]\\
+\(\kappa\,(\beta-1)+(\kappa+\beta-1)\,\frac d{d+2}\)\iM{\frac{|\nabla u|^4}{u^2}}+\frac{\theta\,d}{d-1}\iM{{\Ric}(\nabla u,\nabla u)}\;.
\end{multline*}
Noting again that $[\mathrm L_gu]\cdot g=0$, we can rewrite the first line as
\[
\frac{\theta\,d}{d-1}\left[\iM{\|\,\mathrm L_gu\,\|^2}-\frac2\theta\,\frac{d-1}{d+2}\,(\kappa+\beta-1)\iM{[\mathrm L_gu]\cdot\left[\frac{\nabla u\otimes\nabla u}u-\frac gd\,\frac{|\nabla u|^2}u\right]}\right]\,,
\]
which, by completing the square, turns into
\[
\frac{\theta\,d}{d-1}\iM{\|\mathrm Q_g^\theta u\|^2}-\frac1\theta\(\frac{d-1}{d+2}\)^2(\kappa+\beta-1)^2\iM{\frac{|\nabla u|^4}{u^2}}
\]
because
\[
\iM{\left\|\frac{\nabla u\otimes\nabla u}u-\frac gd\,\frac{|\nabla u|^2}u\right\|^2}=\(1-\frac 1d\)\iM{\frac{|\nabla u|^4}{u^2}}\;.
\]
This completes the proof of Lemma~\ref{Lem:G}.
\end{proof}

At this point let us mention that $\theta$ can be chosen equal to $1$. This was the strategy \cite{Demange-PhD,MR2381156,MR2459454}. This leaves some flexibility on the choice of $\beta$, as it is mentioned in \cite[Section~3.12]{Demange-PhD}. If $1\le d\le4$, a detailed discussion is needed. If $d\ge 5$, we have for instance the following result.
%---------------------------------------------------------------------
\begin{lem}\label{Lem:beta} Assume that $d\ge 5$. If $\theta=1$, then $\mu$ as defined in Lemma~\ref{Lem:G} is non-positive if $\beta$ is such that
\[
\beta_-(p)\le\beta\le\beta_+(p)\quad\forall\,p\in(1,2^*)\,,
\]
where $\beta_\pm:=\frac{\mathsf b\pm\sqrt{\mathsf b^2-\mathsf a}}{2\,\mathsf a}$ with $\mathsf a=2-p+\left[\frac{(d-1)\,(p-1)}{d+2}\right]^2$ and $\mathsf b=\frac{d+3-p}{d+2}$.\end{lem}
%---------------------------------------------------------------------
\begin{proof} Elementary computations show that $\mathsf a$ is positive, $\mathsf b^2-\mathsf a$ is positive for any $p\in(1,2^*)$ and so $\mu(\beta)=\mathsf a\,\beta^2-2\,\mathsf b\,\beta+1$ is negative if $\beta\in(\beta_-(p),\beta_+(p))$, where $\beta_\pm$ are the two roots of the equation $\mu(\beta)=0$. Notice that for $d\le4$, $\mathsf a=0$ or becomes negative for some values of $p$. \end{proof}

See Corollary~\ref{Cor:Theta=1} for some consequences of Lemma~\ref{Lem:beta}. Instead of choosing $\theta=1$, we can take it as small as possible in order that, for some well chosen $\beta$, $\mu$ is still non-positive. The optimal choice of $\mu$ is such that the critical level of the parabola $\beta\mapsto\mu(\beta)$ is exactly $0$. For this value of $\beta$, we can minimize $\theta$ among all possible choices that make $\mu$ non-positive and $\beta$ is determined uniquely, independently of the sign of the term $\beta^2$ in the expression of $\mu$. Let us give details. Up to now, we have shown that
\begin{multline*}
\frac1{2\,\beta^2}\,\frac d{dt}\mathcal F[u]=-\,(1-\theta)\iM{(\Delta_gu)^2}-\frac{\theta\,d}{d-1}\left[\iM{\|\mathrm Q_g^\theta u\|^2}+\iM{{\Ric}(\nabla u,\nabla u)}\right]\\
+\,\mu\iM{\frac{|\nabla u|^4}{u^2}}+\,\lambda\iM{|\nabla u|^2}
\end{multline*}
where, after taking into account \eqref{Eqn:kappa}, $\mu$ is given by
\begin{multline*}
\mu=\frac1\theta\(\frac{d-1}{d+2}\)^2\beta^2\,(p-1)^2-\big(1+\beta\,(p-2)\big)\,(\beta-1)-\beta\,(p-1)\,\frac d{d+2}\\
=\(\frac1\theta\(\frac{d-1}{d+2}\)^2(p-1)^2-(p-2)\)\beta^2-2\,\frac{d+3-p}{d+2}\,\beta+1\;.
\end{multline*}
The goal is to find values of $\beta$ for which~$\mu$ is equal to zero, thus maximizing the other terms. Unless $\theta\,(p-2)=(p-1)^2\,(d-1)^2/(d+2)^2$, the coefficient $\mu$ is quadratic in terms of $\beta$, and we can choose the extremum of $\mu$ with respect to $\beta$, \emph{i.e.},
\[
\beta=\frac{(d+2)\,(d+3-p)\,\theta}{(d-1)^2\,(p-1)^2-(d+2)^2\,(p-2)\,\theta}\;.
\]
With this specific value of $\beta$, we adjust the value of $\theta$ so that $\mu(\beta)=0$, if $p\neq d+3$. Because we assume that $p<\frac{2\,d}{d-2}$ if $d\ge3$, $p=d+3$ cannot occur unless $d=2$ and $p=5$. All computations done, we find that
\be{Eqn:Choice}
\theta=\frac{(d-1)^2\,(p-1)}{d\,(d+2)+p-1}\quad\mbox{and}\quad\beta=\frac{d+2}{d+3-p}
\ee
if $(p,d)\neq(5,2)$ and one can easily check that with these values, and without condition on $\theta$, we actually have $\mu=0$. When $p$ varies in the interval $(1,2^*)$, $\theta=\theta(p)$ ranges from $\theta(1)=0$ to $\lim_{p\to2^*}\theta(p)=1$ if $d\ge2$ and $(p,d)\neq(5,2)$. It is left to the reader to check that the expression of $\mathrm Q_g$ given in Section~\ref{Sec:Intro} coincides with $\mathrm Q_g^\theta$ for the above choices of $\theta$ and $\beta$. Notice that with such a choice of $\theta$, we get $\beta>1$ for any $p\in(1,2^*)$ if $d\ge3$, and any $p\in(1,5)$ if $d=2$. On the contrary, we find that $\beta<0$ if $p>5$ and $d=2$. If $p=5$ and $d=2$, then we have $\mu=\big(\frac1\theta-3\big)\,\beta^2+1$, so that $\theta$ has to be chosen strictly larger than $1/3$.

Collecting the results of Section~\ref{Sec:Flow}, with $\Lambda_\star$ given by \eqref{LambdaStar}, we have found that
%---------------------------------------------------------------------
\begin{prop}\label{Prop:flow} Let $d\ge2$, $p\in(1,2)\cup(2,2^*)$ and assume that $\beta$ and $\theta$ are given by~\eqref{Eqn:Choice} if $p\neq5$ or $d\neq2$. If $p=5$ and $d=2$, we assume that $\theta\in(1/3,1)$ and $\beta=\sqrt{\theta/(3\,\theta-1)}$. With $\kappa$ given by~\eqref{Eqn:kappa}, we get
\[
\frac1{2\,\beta^2}\,\frac d{dt}\mathcal F[u]\le(\lambda-\Lambda_\star)\iM{|\nabla u|^2}
\]
if $u$ is a smooth positive solution of \eqref{flow}.\end{prop}
%---------------------------------------------------------------------
The case $d=1$ has to be handled separately because Lemma~\ref{BW1} does not make sense. However, computations are much easier because
\[
\iM{u''\,\frac{|u'|^2}u}=\frac 13\iM{\frac{|u'|^4}{u^2}}
\]
so that there is no cross term to compute. The reader is invited to check that Lemmata~\ref{Lem:FirstEigenvalue}, \ref{Lem:Lp} and \ref{Lem:F} hold when $d=1$, thus showing that
\[
\frac1{2\,\beta^2}\,\frac d{dt}\mathcal F[u]=-\iM{\left[|u''|^2-\,\mu\,\frac{|u'|^4}{u^2}\right]}+\lambda\iM{|u'|^2}\;.
\]
with $-\mu=\frac 13\,(\kappa+\beta-1)+\kappa\,(\beta-1)=\frac 13\,\beta\,(p-1)+\big(1+\beta\,(p-2)\big)\,(\beta-1)$. Hence, for all $p>1$, there are two values of $\beta$, given by
\be{Eqn:BetaPM}
\beta_\pm:=\frac{p-4\pm\sqrt{(p-1)\,(p+2)}}{3\,(p-2)}\;,
\ee
such that $\mu=0$. For $\beta=\beta_\pm$, we have found that
\[
\frac1{2\,\beta^2}\,\frac d{dt}\mathcal F[u]=-\iM{|u''|^2}+\lambda\iM{|u'|^2}\le0
\]
for any $\lambda\le\lambda_1$.
%---------------------------------------------------------------------
\begin{prop}\label{Prop:flow-d=1} Let $d=1$, $p\in(1,2)\cup(2,+\infty)$. Assume that $\kappa$ and $\beta$ are given respectively by~\eqref{Eqn:kappa} and by~\eqref{Eqn:BetaPM}. Then
\[
\frac1{2\,\beta^2}\,\frac d{dt}\mathcal F[u]\le(\lambda-\lambda_1)\iM{|\nabla u|^2}
\]
if $u$ is a smooth positive solution of \eqref{flow} and
\begin{enumerate}
\item[(i)] either $p>2$ and $\beta\in(-\infty,\beta_-]\cup[\beta_+,+\infty)$,
\item[(ii)] or $p\in(1,2)$ and $\beta\in[\beta_-,0)\cup(0,\beta_+]$.
\end{enumerate}\end{prop}
%---------------------------------------------------------------------

\medskip In order to justify the above computations, we have to establish that the solution of \eqref{flow} is smooth and converges to a constant. Since this is not at all the scope of this paper, let us simply sketch that strategy of proof. A positive solution of \eqref{Eqn} of class $C^2$ is actually smooth by standard elliptic theory, and thus we know that the initial datum for \eqref{flow} is smooth. If we simply consider an optimal function for \eqref{Ineq:Interp}, we may proceed by regularization and approximate the solution by a bounded solution, which is also positive, bounded away from zero. A further regularization of the initial datum allows to get back to the setting of Theorems~\ref{Thm:LV}-\ref{Thm:Main2}. By applying the Maximum Principle, we can then extend the boundedness properties (from above and from below) of the initial datum to the solution of \eqref{flow} for any positive time. Regularity results of Harnack and Aronson-B\'enilan type like the ones in \cite{MR2487898} and subsequent papers then hold. Higher order regularity results follow: see for instance \cite{MR2282669} in the Euclidean case. The delayed regularity phenomenon can be avoided by starting with smooth enough initial data.

It remains to show that the solution converges to a constant. One can argue as follows. Using Lemma~\ref{Lem:Lp} and H\"older's inequality, we know that $\iM{u^{2\,\beta}}$ is bounded uniformly in $t$ from above if $p>2$ and from below if $p\in(1,2)$. By \eqref{Eqn:L2}, is bounded uniformly in $t$ and, as a consequence, $\int_0^{+\infty}\(\iM{|\nabla u|^2}\)dt$ is finite. Standard arguments allow then to prove that $\lim_{t\to+\infty}\iM{|\nabla u|^2}=0$.

%%%%%%%%%%%%%%%%%%%%%%%%%%%%%%%%%%%%%%%%%%%%%%%%%%%%%%%%%%%%%%%%%%%%%%
%%%%%%%%%%%%%%%%%%%%%%%%%%%%%%%%%%%%%%%%%%%%%%%%%%%%%%%%%%%%%%%%%%%%%%
\section{Proofs}\label{Sec:Proofs}

Let us consider $\Lambda_\star$ and $\theta$ as defined by \eqref{LambdaStar} and \eqref{Eqn:Choice}.
%---------------------------------------------------------------------
\begin{prop}\label{Prop:Estimates} Using the notations of Section~\ref{Sec:Intro}, we have the estimates
\[
(1-\theta)\,\lambda_1+\theta\,\frac{d\,\rho}{d-1}\le\lambda_\star\le\Lambda_\star\le\lambda_1\;.
\]\end{prop}
%---------------------------------------------------------------------
\begin{proof} The first inequality is a consequence of Lemma~\ref{Lem:FirstEigenvalue}. For the second one, we observe that
\[
\lim_{\eps\to0}\frac1{\eps^2}\,\|\mathrm Q_g^\theta(1+\eps\,v)\|^2=\|\mathrm L_gv\|^2\,.
\]
Using $u=1+\eps\,v$ as a test function, we find that
\[
\Lambda_\star\le\inf_{v\in\H^2(\M)\setminus\{0\}}\frac{\iM{(\Delta_gv)^2}}{\iM{|\nabla v|^2}}=\lambda_1
\]
according to Lemma~\ref{BW1} and Lemma~\ref{Lem:FirstEigenvalue}.\end{proof}

\begin{proof}[Proof of Theorem~\ref{Thm:Main2}] If $v=u^\beta$ is a solution to~\eqref{Eqn}, it is straightforward to check that
\[
\Lap u+(\beta-1)\,\frac{|\nabla u|^2}u-\frac\lambda{\beta\,(p-2)}\,(u-u^\kappa)=0\;.
\]
Using the fact that $\iM{u^\kappa\(\Lap u+\kappa\,\frac{|\nabla u|^2}u\)}=0$, it follows by Proposition~\ref{Prop:flow} that at $t=0$ we have
\begin{multline*}
0=-\iM{\(\Lap u+(\beta-1)\,\frac{|\nabla u|^2}u-\frac\lambda{\beta\,(p-2)}\,(u-u^\kappa)\)\(\Lap u+\kappa\,\frac{|\nabla u|^2}u\)}\\
=\frac1{2\,\beta^2}\,\frac d{dt}\mathcal F[u]\le(\lambda-\Lambda_\star)\iM{|\nabla u|^2}\;,
\end{multline*}
so that $u$ is constant if $\lambda<\Lambda_\star$.\end{proof}

\begin{proof}[Proof of Theorem~\ref{Thm:Inequality}] We apply Proposition~\ref{Prop:flow} to $u(t=0,\cdot)=v^{1/\beta}$. For $\lambda\le\Lambda_\star$, $\mathcal F[u]$ is non-increasing. We know that as $t\to\infty$ $u$ and $\nabla u$ converge respectively to a constant and to $0$, which proves that $\mathcal F[u(t=0,\cdot)]\ge\mathcal F[u(t,\cdot)]\ge0=\lim_{t\to\infty}\mathcal F[u(t,\cdot)]$ for any $t\ge0$. Indeed, by considering the functional $\mathcal F$ for some $\lambda<\Lambda_\star$, we know that $\int_0^\infty\iM{|\nabla u|^2}\;dt$ is finite, which shows that $u(t,\cdot)$ converges to a positive constant $c$ as $t\to\infty$. This convergence is in a weak sense, but classical results for porous medium equations apply and it is not difficult to obtain the uniform convergence. See~\cite{MR2282669,Vazquez07} and references therein for results concerning the porous media. Then standard parabolic methods apply and convergence holds in $C^k(\M)$, for an arbitrary $k\ge1$.

Hence, for any $\lambda\le\Lambda_\star$, this proves that $\mathcal F[u]$ is non-negative for any $u$, \emph{i.e.}, that \eqref{Ineq:Interp} holds. It remains to consider the case $\lambda=\Lambda_\star<\lambda_1$, for which we do not have a rigidity result. Let us prove that optimal functions in \eqref{Ineq:Interp}, that is, functions for which the inequality becomes an equality, are all constants. Let~$v$ be an optimal function for~\eqref{Ineq:Interp}. We can assume that $\nrm vp=1$ with no restriction and if $u$ is the solution of~\eqref{flow} with initial datum $v^{1/\beta}$, then
\[
\mathcal F[u]=\int_0^\infty\left[(1-\theta)\iM{(\Delta_gu)^2}+\frac{\theta\,d}{d-1}\iM{\Big[\,\|\mathrm Q_g u\|^2+\Ric(\nabla u,\nabla u)\Big]}-\,\Lambda_\star\iM{|\nabla u|^2}\right]dt=0\;,
\]
where equality holds for smooth solutions, but can also be obtained by a regularization procedure in the other cases. If we write that $u(t,x)=c+\varepsilon\,\varphi(t,x)$, where $\varepsilon=\varepsilon(t)\to0$ is adjusted so that $\nrM{\nabla\varphi}2^2+\nrM{\Lap\varphi}2^2=1$, then we can write as in the proof of Lemma~\ref{BW2} that
\[
\iM{\Delta_gu\,\frac{|\nabla u|^2}u}=\iM{\frac{|\nabla u|^4}{u^2}}-2\iM{[\mathrm H_g u]\cdot \left[\frac{\nabla u \otimes \nabla u}u\right]}
\]
after one integration by parts. On the one hand we have
\begin{eqnarray*}
&&\iM{\Delta_gu\,\frac{|\nabla u|^2}u}=\varepsilon^3\iM{\Delta_g\varphi\,\frac{|\nabla \varphi|^2}c}\big(1+O(\varepsilon)\big)\,,\\
&&\iM{[\mathrm H_g u]\cdot \left[\frac{\nabla u \otimes \nabla u}u\right]}=\varepsilon^3\iM{[\mathrm H_g \varphi]\cdot \left[\frac{\nabla \varphi \otimes \nabla \varphi}c\right]}\big(1+O(\varepsilon)\big)\,,
\end{eqnarray*}
and on the other hand, using again an integration by parts, we know that
\[
\iM{\Delta_g\varphi\,|\nabla \varphi|^2}=-\,2\,\iM{[\mathrm H_g \varphi]\cdot \left[\nabla \varphi \otimes \nabla \varphi\right]}\,.
\]
This shows that
\[
\iM{\frac{|\nabla u|^4}{u^2}}=O(\varepsilon^4)\,,
\]
so that $\|\mathrm Q_g u\|^2\sim\|\mathrm L_g u\|^2$ as $t\to+\infty$. Arguing as in Proposition~\ref{Prop:Estimates} and using Lemma \ref{BW1}, we find that
\begin{multline*}
0=(1-\theta)\iM{(\Delta_gu)^2}+\frac{\theta\,d}{d-1}\iM{\Big[\,\|\mathrm Q_g u\|^2+\Ric(\nabla u,\nabla u)\Big]}-\,\Lambda_\star\iM{|\nabla u|^2}\\
\sim[(1-\theta)\iM{(\Delta_gu)^2}+\frac{\theta\,d}{d-1}\iM{\Big[\,\|\mathrm L_g u\|^2+\Ric(\nabla u,\nabla u)\Big]}-\,\Lambda_\star\iM{|\nabla u|^2}\\
=\iM{(\Delta_gu)^2}-\,\Lambda_\star\iM{|\nabla u|^2}\ge\(\lambda_1-\,\Lambda_\star\)\iM{|\nabla u|^2}\;,
\end{multline*}
thus obtaining an obvious contradiction if $\Lambda_\star<\lambda_1$, unless $u$ is constant. Hence, in that case, we know that all minimizers of~$\mathcal F$ are constants when $\lambda=\Lambda_\star$.

Now let us consider the optimal constant $\Lambda$ in \eqref{Ineq:Interp}. If $\Lambda=\Lambda_\star$, consider a sequence $(\lambda_n)_{n\in\N}$ such that $\lambda_n>\Lambda_\star$ for any $n\in\N$, $\lim_{n\to\infty}\lambda_n=\Lambda_\star$ and a sequence $(v_n)_{n\in\N}$ of functions in $\H^1(\M)$ such that $\nrm{v_n}2=1$ for any $n\in\N$ and $\mathcal F_{\lambda_n}[v_n]<0$ where, with obvious notations, $\mathcal F_{\lambda_n}$ denotes the functional $\mathcal F$ written for $\lambda=\lambda_n$. Actually we can consider minimizers of $\mathcal F_{\lambda_n}$ in $\{u\in\H^1(\M)\,:\,\nrm u2=1\}$. Let us consider the quotient
\[
\mathcal Q[v]:=\frac{(p-2)\,\nrm{\nabla v}2^2}{\nrm vp^2-\nrm v2^2}
\]
and observe that the optimal constant $\Lambda$ is such that
\[
\Lambda=\inf_{v\in\H^1(\M)\setminus\{0\}}\mathcal Q[v]=\lim_{n\to\infty}\mathcal Q[v_n]\;.
\]
Up to the extraction of a subsequence, $\seq vn$ converges to a limit $v$ which solves the Euler-Lagrange equation~\eqref{Eqn}. Multiplying the equation by $v$ and integrating, we see that $v$ is a minimizer of $\mathcal F_{\Lambda_\star}$ and is therefore constant. Because of the constraint $\nrm {v_n}2=1$, the limit is $1$. Let us define $w_n:=\eps_n\,(v_n-1)$ with $\eps_n:=\nrm{v_n-1}2\to 0$. A standard computation then shows that
\[
\mathcal Q[v_n]=\mathcal Q[1+\eps_n\,w_n]\sim\frac{\nrm{\nabla w_n}2^2}{\nrm{w_n}2^2}
\]
from which we deduce that $\Lambda_\star\ge\lambda_1$, a contradiction. Hence $\Lambda_\star<\lambda_1$ means that $\Lambda>\Lambda_\star$, which concludes the proof.
\end{proof}

For completeness, let us come back to the case $\theta=1$ considered in Lemma~\ref{Lem:beta}. It is striking that various flows can be used for the same inequality (see \cite[Section~3.12]{Demange-PhD} and \cite{DEKL}). There is a whole interval of values of $\beta$ which makes $\mu$ defined as in Lemma~\ref{Lem:beta} negative if we choose $\theta=1$. As a consequence, we have the following result.
%---------------------------------------------------------------------
\begin{cor}\label{Cor:Theta=1} Assume that $d\ge5$ and $p\in(1,2^*)$. For any $\beta\in(\beta_-(p),\beta_+(p))$, then we have
\[
\frac d{dt}\mathcal F[u]\le0
\]
if $u$ is a solution of \eqref{flow}.\end{cor}
%---------------------------------------------------------------------
If we come back to the strategy of \cite{MR2381156,MR2459454}, the specific choice that was made there corresponds to $\beta=\frac4{6-p}$ and in the critical case, we find that $\beta=\beta_\pm(2^*)$ for any $d\ge3$. However, although admissible, such a choice of $\beta$ is not optimal in the subcritical range $p\in(2,2^*)$: if we minimize $\theta$, that is, if we optimize on the estimate of the constant in the interpolation inequality~\eqref{Ineq:Interp} for $p\in(2,2^*)$, $\beta=\frac4{6-p}$ is not the best possible choice. Moreover, this choice is not admissible in the whole range $p\in(1,2)$.

%%%%%%%%%%%%%%%%%%%%%%%%%%%%%%%%%%%%%%%%%%%%%%%%%%%%%%%%%%%%%%%%%%%%%%
%%%%%%%%%%%%%%%%%%%%%%%%%%%%%%%%%%%%%%%%%%%%%%%%%%%%%%%%%%%%%%%%%%%%%%
\section{Improved results and consequences}\label{Sec:CNS}

Theorem~\ref{Thm:LV} is a consequence of Proposition~\ref{Prop:Estimates}. The expression of $\Lambda_\star$ in Theorem~\ref{Thm:Main2} is not as simple as the one of $\lambda_\star$ in Theorem~\ref{Thm:Main} but provides an improved condition for \emph{rigidity}. Up to now, we have dealt only with power law nonlinearities. It is not difficult to generalize our results to nonlinearities that compare with power laws, for instance $f(v)=|v|^{p-1}\,v+|v|^{q-1}\,v$ for some $q>p$. Our first extension of Theorem~\ref{Thm:Main2} goes as follows.
%---------------------------------------------------------------------
\begin{thm}\label{Thm:General} Let $f$ be a Lipschitz increasing function such that
\[
\frac1{p-2}\(f'(v)-(p-1)\,\frac{f(v)}v\)\le 0\quad\forall\,v>0\;\mbox{a.e.}
\]
Then under the same conditions as in Theorem~\ref{Thm:Main2}, for any $\lambda\in(0,\Lambda_\star)$, any smooth positive solution to the equation
\[
-\,\Lap v+\frac\lambda{p-2}\,\big(v-f(v)\big)=0
\]
is equal to a constant $c$ which satisfies $f(c)=c$.
\end{thm}
%---------------------------------------------------------------------
This is a characterization of the $C^2(\M)$ solutions, which are all constant if they exist. In particular, if the equation $f(c)=c$ has no solution in $\R$, this means that the above equation has no smooth positive solution.

\begin{proof}[Proof of Theorem~\ref{Thm:General}] As in the proof of Theorem~\ref{Thm:Main2} the function $u$ such that $v=u^\beta$ is a solution to
\[
-\Lap u-(\beta-1)\,\frac{|\nabla u|^2}u+\frac\lambda{\beta\,(p-2)}\(u-\frac{f(v)}v\,u\)=0\;.
\]
We may multiply the equation by $\(\Lap u+\kappa\,\frac{|\nabla u|^2}u\)$ and get
\begin{multline*}
0=-\iM{\(\Lap u+(\beta-1)\,\frac{|\nabla u|^2}u-\frac\lambda{\beta\,(p-2)}\,u\)\(\Lap u+\kappa\,\frac{|\nabla u|^2}u\)}\\
+\frac\lambda{p-2}\iM{\(f'(v)-(p-1)\,\frac{f(v)}v\)}
\le(\lambda-\Lambda_\star)\iM{|\nabla u|^2}\;,
\end{multline*}
where the inequality holds by our assumption on $f$ and by the same computations as in the proof of Theorem~\ref{Thm:Main2}. So, if $\lambda<\Lambda_\star$, $u$ is constant.\end{proof}

Let us next define
\be{Def:F}
F[v](x):=\nrM vp^p+p\int_{\nrM vp}^{v(x)}f(s)\;ds\,.
\ee
This functional enjoys the following properties:
\begin{enumerate}
\item[(i)] If $c$ is a positive real constant, $F[c]=c^p$.
\item[(ii)] If $f(s)=s^{p-1}$ for any $s>0$, then $\iM{F[v]}=\nrM vp^p$.
\end{enumerate}
The strategy for writing a functional inequality is essentially the same as the one of Theorem~\ref{Thm:Main2}, except that
\[
\iM{F[v]}=\iM{F[u^\beta]}
\]
is not preserved by the flow defined by \eqref{flow}. However, if $\lambda\in(0,\Lambda_\star)$, the functional
\[
\textstyle\mathcal F[u]:=\nrm{\nabla(u^\beta)}2^2-\frac{\lambda}{p-2}\,\left[{\(\iM{F[u^\beta]}\)}^{2/p}-\nrm{u^\beta}2^2\right]
\]
is still non-increasing because, with $v=u^\beta$, we have
\begin{multline*}
\frac1{p-2}\,\frac d{dt}\iM{F[v]}=\frac{\beta\,p}{p-2}\iM{\left[(\beta+\kappa-1)\,\frac{f(v)}v-\beta\,f'(v)\right]|\nabla u|^2}\\
= \frac{\beta^2\,p}{p-2}\iM{\left[(p-1)\,\frac{f(v)}v-f'(v)\right]|\nabla u|^2}\ge0\;.
\end{multline*}

Under some additional coercivity assumptions, \emph{e.g.}
\be{Eqn:Coercivity}
f(0)=0\quad\mbox{and}\quad\left\{\begin{array}{ll}
\lim_{s\to\infty}\frac{f(s)}s=\infty\quad&\mbox{if}\quad p>2\,,\\[6pt]
\lim_{s\to\infty}\frac{f(s)}s=0\quad&\mbox{if}\quad p<2\,,
\end{array}\right.
\ee
the functional $\mathcal F[u]$ is bounded from below, $\lim_{t\to\infty}\nrM{\nabla u(t,\cdot)}2=0$, and $\lim_{t\to\infty}\mathcal F[u(t,\cdot)]=0$. Hence we also have the inequality
\[
\nrm{\nabla v}2^2-\frac{\lambda}{p-2}\,\left[{\(\iM{F[v]}\)}^{2/p}-\nrm v2^2\right]\ge0\quad\forall\,v\in\H^1(\M)
\]
for any $\lambda\in(0,\Lambda_\star)$. Here $F$ is given as above by \eqref{Def:F}. Our method even provides an integral remainder term that can be computed using the flow, when $\lambda=\Lambda_\star$. If $u$ a smooth solution of~\eqref{flow} with initial datum $v^{1/\beta}$, let
\begin{multline*}
\mathcal R[v]:=\beta^2\int_0^\infty\left[(1-\theta)\iM{(\Delta_gu)^2}+\frac{\theta\,d}{d-1}\iM{\Big[\,\|\mathrm Q_g^\theta u\|^2+\Ric(\nabla u,\nabla u)\Big]}-\Lambda_\star\iM{|\nabla u|^2}\right]dt\\
+\frac{\Lambda_\star\,\beta^2}{p-2}\int_0^\infty\left[{\(\iM{F[u^\beta]}\)}^{\frac 2p-1}\iM{\left[(p-1)\,\frac{f(u^\beta)}{u^\beta}-\,f'(u^\beta)\right]|\nabla u|^2}\right]\,dt\;.
\end{multline*}
%---------------------------------------------------------------------
\begin{cor}\label{Cor:InequalityRemainder} Assume that the assumptions of Theorem~\ref{Thm:General} and \eqref{Eqn:Coercivity} hold. With the same notations as above and $\Lambda_\star$ defined by \eqref{LambdaStar}, for any $p\in(1,2)\cup(2,2^*]$ if $d\ge3$ and any $p\in(1,2)\cup(2,\infty)$ otherwise, for any smooth positive function $v$ on $\M$, we have the inequality
\[
\nrm{\nabla v}2^2-\frac{\Lambda_\star}{p-2}\,\left[\(\iM{F[v]}\)^{2/p}-\nrm v2^2\right]\ge2\,\mathcal R[v]\;.
\]
\end{cor}
%---------------------------------------------------------------------
Notice that the coercivity assumption~\eqref{Eqn:Coercivity} shows that $\mathcal F$ is bounded from below. Because of the flow, we know that the solution $u$ of \eqref{flow} is such that $\nrM{\nabla u}2$ converges to $0$ a.e.~in $t$, as $t\to\infty$, so that $\lim_{t\to\infty}\mathcal F[u(t,\cdot)]=0$. Details of the proof are left to the reader.

\medskip Let us conclude this section by some comments in the case $p\in(1,2)$ and $d\ge2$. By the definition~\eqref{LambdaStar} of $\Lambda_\star$, we know that
\[
\Lambda_\star\ge(1-\theta)\,\lambda_1+\theta\,\rho_\star\quad\mbox{with}\quad\rho_\star:=\frac d{d-1}\,\inf_{u\in\H^2(\M)\setminus\{0\}}\frac{\iM{\Big[\,\|\mathrm Q_g u\|^2+\Ric(\nabla u,\nabla u)\Big]}}{\iM{|\nabla u|^2}}\;.
\]
We also know that $\rho_\star\ge\rho$ and hence $\Lambda_\star\ge\lambda_\star$. Let us denote by $\Lambda(p)$ the optimal constant in~\eqref{Ineq:Interp}. From Theorem~\ref{Thm:Inequality} and Proposition~\ref{Prop:Estimates}, we have
\[
\lambda_1\ge\Lambda(p)\ge\Lambda_\star(p)\ge(1-\theta)\,\lambda_1+\theta\,\rho_\star
\]
with
\[
\theta=\frac{(d-1)^2\,\eta}{d\,(d+2)+\eta}\;,\quad\eta=p-1\;.
\]
If $p=1$, we have $\Lambda(1)=\lambda_1$.

On the other hand, passing to the limit as $p\to2$ is straightforward so that we get
\[
\frac12\,\Lambda(2)\iM{|v|^2\,\log\(\frac{|v|^2}{\nrm v2^2}\)}\le\nrm{\nabla v}2^2\quad\forall\,v\in\H^1(\M)\;.
\]
Here $\Lambda(2)$ denotes the optimal constant and we have $\Lambda(2)\ge\Lambda_\star(2)\ge(1-\theta)\,\lambda_1+\,\theta\,\rho_\star$ with $\theta=((d-1)/(d+1))^2$. The interpolation method based on spectral estimates of \cite{MR954373} and later extended in \cite[Theorem~Ê2.4]{ABD}(also see \cite[Section~2.2]{DEKL}) shows that
\[
\Lambda(p)\ge\lambda_1\,\frac{1-\eta}{1-\eta^\alpha}\quad\mbox{with}\quad\alpha=\lambda_1/\Lambda(2)\quad\mbox{and}\quad\eta=p-1\;.
\]
Such an estimate is better than $\Lambda(p)\ge(1-\theta)\,\lambda_1+\theta\,\rho_\star$ at least in a neighborhood of $p=2$ if $\Lambda(2)>(1-\theta_2)\,\lambda_1+\theta_2\,\rho_\star$ with $\theta_2=((d-1)/(d+1))^2$. This proves that for estimates on the optimal constant in~\eqref{Ineq:Interp}, there is space for improvements.

%%%%%%%%%%%%%%%%%%%%%%%%%%%%%%%%%%%%%%%%%%%%%%%%%%%%%%%%%%%%%%%%%%%%%%
%%%%%%%%%%%%%%%%%%%%%%%%%%%%%%%%%%%%%%%%%%%%%%%%%%%%%%%%%%%%%%%%%%%%%%
\section{Concluding remarks}\label{Sec:Conclusion}

All computations in this paper are based on the action of the flow~\eqref{flow} on the functional $\mathcal F$. The functional $\mathcal F$ is monotonously decaying with respect to time for the right choice of the exponent $\beta$ (and the optimal choice is therefore prescribed by the method) while the solution of~\eqref{flow} converges to a constant. The method gives a meaningful way of reinterpreting the computations of \cite{MR1134481,MR615628,MR1338283,MR1412446,MR1631581} and suggest the improvements of Theorems~\ref{Thm:Main}, \ref{Thm:Main2} and \ref{Thm:Inequality}. However, the use of the flow itself is not mandatory except in the proof of Theorem~\ref{Thm:Inequality} to handle the case $\lambda=\Lambda_\star$ and in the extensions corresponding to Theorem~\ref{Thm:General} and Corollary~\ref{Cor:InequalityRemainder}. The flow approach however paves the road for the various improvements presented in this paper: the integral criteria defining $\lambda_\star$ and $\Lambda_\star$ or the extension to general nonlinearities as in Theorem~\ref{Thm:General}. In Theorem~\ref{Thm:Inequality}, an alternative approach could be used to directly prove that
\[
\iM{(\Delta_gu)^2}+\frac{\theta\,d}{d-1}\iM{\Big[\,\|\mathrm Q_g u\|^2+\Ric(\nabla u,\nabla u)\Big]}-\,\Lambda_\star\iM{|\nabla u|^2}=0
\]
means that $u$ is a constant if $\Lambda_\star<\lambda_1$. Such a result is a consequence of our method, at least when $u$ is a solution of~\eqref{flow} whose initial datum is a minimizer of $\mathcal F$.

The choice~\eqref{Eqn:Choice} is the one that minimizes the value of $\theta$ and therefore maximizes the value of $\Lambda_\star$. However, if we relax a little bit the condition on $\theta\in(0,1]$ and take $\theta$ such that
\[
\frac{(d-1)^2\,(p-1)}{d\,(d+2)+p-1}<\theta<1\;,
\]
then there is a whole interval $[\beta_-,\beta_+]$ of values of $\beta$ for which the method of Section~\ref{Sec:Flow} still works, \emph{i.e.}, for which $\mu$ is non-positive, for any $p\in(1,2)\cup(2,2^*)$. The extreme case $\theta=1$ has been considered in Lemma~\ref{Lem:beta} and Corollary~\ref{Cor:Theta=1}. If $d\ge 3$ and $p=2^*$, it follows from \eqref{Eqn:Choice} that $\theta=1$ and $\beta$ is uniquely defined. Consistently, we observe that $\lim_{p\to2^*}(\beta_+(p)-\beta_-(p))=0$, so that the interval shrinks to a point when $p$ approaches the critical value $2^*$, $d\ge3$. Alternatively, there are values of $\beta$ for which a whole interval in terms of $p$ can be covered, as it has been noticed in \cite{DEKL}, in the case of the sphere when $\beta=1$: in such a case, the method applies for any $p\in(1,2)\cup(2,2^\sharp]$, where $2^\sharp:=(2\,d^2+1)/(d-1)^2$.

Our paper leaves open the issue of deciding under which conditions the interval $(\Lambda,\lambda_1)$ is empty or not. In other words, are there non-constant minimizers of $\mathcal F$ for some $\lambda<\lambda_1$ ? If $\Lambda_*<\lambda_1$, rigidity results of Theorem~\ref{Thm:Main2} and optimality results of Theorem~\ref{Thm:Inequality} are apparently of different nature, although they coincide in the case of the sphere. It is not difficult to produce examples of manifolds for which $\Lambda_*<\lambda_1$ (take non-constant curvatures and consider the case of $p$ in a neighborhood of $2^*$ when $d\ge3$). However, in this paper we are dealing only with sufficient conditions and it is therefore not possible to conclude that criteria for rigidity are stronger than criteria for determining the optimal constant in the interpolation Inequality~\eqref{Ineq:Interp}.

%%%%%%%%%%%%%%%%%%%%%%%%%%%%%%%%%%%%%%%%%%%%%%%%%%%%%%%%%%%%%%%%%%%%%%
%%%%%%%%%%%%%%%%%%%%%%%%%%%%%%%%%%%%%%%%%%%%%%%%%%%%%%%%%%%%%%%%%%%%%%
\smallskip\noindent{\bf Acknowledgements.} {\small J.D.~and M.J.E.~have been partially supported by ANR grants \emph{CBDif} and \emph{NoNAP}. They are also participating to the MathAmSud project \emph{QUESP} and thank A.~Laptev, F.~Robert, E.~Hebey, M.~Ledoux and C.~Villani for fruitful discussions. J.D.~and M.L.~have been supported in part respectively by ANR grants \emph{STAB} and NSF grant DMS-0901304. The author warmfully thank a referee for detailed comments and suggestions.}
\\[6pt]
{\sl\small\copyright~2014 by the authors. This paper may be reproduced, in its entirety, for non-commercial purposes.}
%%%%%%%%%%%%%%%%%%%%%%%%%%%%%%%%%%%%%%%%%%%%%%%%%%%%%%%%%%%%%%%%%%%%%%
%%%%%%%%%%%%%%%%%%%%%%%%%%%%%%%%%%%%%%%%%%%%%%%%%%%%%%%%%%%%%%%%%%%%%%
%\newpage\nocite*
%\bibliographystyle{siam}\small\bibliography{References}
%\end{document}

%%%%%%%%%%%%%%%%%%%%%%%%%%%%%%%%%%%%%%%%%%%%%%%%%%%%%%%%%%%%%%%%%%%%%%
%%%%%%%%%%%%%%%%%%%%%%%%%%%%%%%%%%%%%%%%%%%%%%%%%%%%%%%%%%%%%%%%%%%%%%
\end{document}